\documentclass[12pt,a4paper,oneside]{amsart}

\usepackage{amsmath} 
\usepackage{amssymb} 
\usepackage{amsthm} 
\usepackage[english]{babel} 
\usepackage[font=small,justification=centering]{caption} 
\usepackage[nodayofweek]{datetime}
\usepackage[shortlabels]{enumitem} 
\usepackage[T1]{fontenc} 
\usepackage[a4paper]{geometry} 
\usepackage[utf8]{inputenc} 
\usepackage{ifthen} 
\usepackage{mathabx} 
\usepackage{mathtools} 
\usepackage[dvipsnames]{xcolor} 
\usepackage[pdftex,  colorlinks=true]{hyperref} 
    \hypersetup{urlcolor=RoyalBlue, linkcolor=RoyalBlue,  citecolor=black}
\usepackage{setspace} 
\usepackage{tikz-cd} 
\usepackage{xfrac} 

\usepackage[style=alphabetic, citestyle=alphabetic, sorting=nyvt]{biblatex}
\usepackage{csquotes}

\AtEveryBibitem{\ifentrytype{article}{\clearfield{issn}}{}}
\AtEveryBibitem{\ifentrytype{article}{\clearfield{url}}{}}
\AtEveryBibitem{\ifentrytype{incollection}{\clearfield{url}}{}}
\AtEveryBibitem{\ifentrytype{book}{\clearfield{url}}{}}

\NewBibliographyString{toappear}
\DefineBibliographyStrings{english}{%
  toappear = {to appear},
}

\renewbibmacro*{in:}{%
  \iffieldundef{pubstate}
    {}
    {\printfield{pubstate}%
     \setunit{\addspace}%
     \clearfield{pubstate}}%
  \printtext{%
    \bibstring{in}\intitlepunct}}

\makeatletter
\def\@tocline#1#2#3#4#5#6#7{\relax
  \ifnum #1>\c@tocdepth 
  \else
    \par \addpenalty\@secpenalty\addvspace{#2}%
    \begingroup \hyphenpenalty\@M
    \@ifempty{#4}{%
      \@tempdima\csname r@tocindent\number#1\endcsname\relax
    }{%
      \@tempdima#4\relax
    }%
    \parindent\z@ \leftskip#3\relax \advance\leftskip\@tempdima\relax
    \rightskip\@pnumwidth plus4em \parfillskip-\@pnumwidth
    #5\leavevmode\hskip-\@tempdima
      \ifcase #1
       \or\or \hskip 1em \or \hskip 2em \else \hskip 3em \fi%
      #6\nobreak\relax
    \dotfill\hbox to\@pnumwidth{\@tocpagenum{#7}}\par
    \nobreak
    \endgroup
  \fi}
\makeatother

\newtheorem{thm}{Theorem}[section]
\newtheorem{lemma}[thm]{Lemma}
\newtheorem{corollary}[thm]{Corollary}
\newtheorem{prop}[thm]{Proposition}

\newtheorem{question}[thm]{Question}

\newtheorem{thmx}{Theorem}

\theoremstyle{definition}

\newtheorem{remark}[thm]{Remark}

\DeclareMathOperator{\Aut}{\mathrm{Aut}}

\DeclareMathOperator{\Comm}{\mathrm{Comm}}
\DeclareMathOperator{\Isom}{\mathrm{Isom}}

\newcommand{\calc}{{\mathcal{C}}}

\newcommand{\calt}{{\mathcal{T}}}

\newcommand*{\fraks}{\mathfrak{S}}

\newcommand{\OO}{\mathrm{O}}

\newcommand{\CAT}{\mathrm{CAT}}

\newcommand*{\nest}{\sqsubseteq}
\newcommand*{\pnest}{\sqsubset}

\newcommand*{\trans}{\pitchfork}

\newcommand{\EE}{\mathbb{E}}

\newcommand{\ZZ}{\mathbb{Z}}

\usepackage{tikz}
\usetikzlibrary{arrows,quotes}
\tikzstyle{blackNode}=[fill=black, draw=black, shape=circle]

\usepackage[foot]{amsaddr}
\title[Lattices, hierarchical hyperbolicity, and virtual torsion-freeness]{Lattices in a product of trees, hierarchically hyperbolic groups, and virtual torsion-freeness}
\author{Sam Hughes}
\email{sam.hughes@maths.ox.ac.uk}
\address{Mathematical Institute, Andrew Wiles Building, University of Oxford, Oxford OX2 6GG, UK}
\subjclass{20F67, 20F65}
\date{\today}

\begin{document}

\maketitle

\begin{abstract}
    We construct cocompact lattices in a product of trees which are not virtually torsion-free.  This gives the first examples of hierarchically hyperbolic groups which are not virtually torsion-free.
\end{abstract}

\section{Introduction}
Hierarchically hyperbolic groups (HHGs) and spaces (HHSs) were introduced by Behrstock, Hagen and Sisto in \cite{BehrstockHagenSisto2017a}.
Hierarchically hyperbolic groups are known to have a number of properties such as having finite asymptotic dimension \cite[Theorem~A]{BehrstockHagenSisto2017}, having a uniform bound on the conjugator length of Morse elements \cite{AbbottBehrstock2019}, and for virtually torsion-free HHGs, their uniform exponential growth is well understood \cite{AbbottNgSpriano2019}.  HHGs belong to the class of semihyperbolic groups \cite[Corollary~F]{HaettelHodaPetyt2020} (see also \cite{DurhamMinskySisto2020}).  In particular, their finitely generated abelian subgroups are undistorted, they have solvable conjugacy problem, finitely many conjugacy classes of finite subgroups, and are of type $\rm{FP}_\infty$.

That HHGs have only many finitely many conjugacy classes of finite subgroups implies that every residually finite HHG is in fact virtually torsion-free.  This motivates the question of whether there exist any HHGs which are not virtually torsion-free.  The question is of considerable interest to specialists since, for example, a number of theorems about HHGs require the assumption of virtual torsion-freeness (see for instance \cite[Theorem~1.1]{AbbottNgSpriano2019} and \cite[Theorem~1.2(3')]{RobbioSpriano2020}).

In this paper we construct an infinite family of $\CAT(0)$ lattices acting faithfully and geometrically on a product of trees.  We then prove that each lattice $\Gamma$ is a hierarchically hyperbolic group and has no finite index torsion-free subgroups.  This appears to be the first examples in the literature of cocompact lattices in a product of trees which are not virtually torsion-free (non-cocompact examples were given by Caprace and R\'emy \cite{CapraceRemy2009}).  

\begin{thmx}[Theorem~\ref{cor.HHGsNotVirTorFree}]\label{corx.HHGsNotVirTorFree}
There exist uniform lattices in products of trees which are hierarchically hyperbolic groups and which are not virtually torsion-free.
\end{thmx}

To the author's knowledge this is the first explicit example of an HHG which is not virtually torsion-free.    The author suspects that it is possible to apply the results of Hagen--Susse \cite{HagenSusse2020} to Wise's examples in \cite{Wise2007} to obtain an HHG which is not virtually torsion-free; however, the construction presented here is much more elementary and gives an explicit HHG structure.

\subsection*{Acknowledgements}
The author was supported by the Engineering and Physical Sciences Research Council grant number 2127970.  This work has received funding from the European Research Council (ERC) under the European Union’s Horizon 2020 research and innovation programme (Grant agreement No. 850930).  This paper contains material from the author's PhD thesis \cite{HughesThesis}.  The author would like to thank his PhD supervisor Ian Leary for his guidance and support.  Additionally, the author would like to thank Mark Hagen, Ashot Minasyan, Harry Petyt, and Motiejus Valiunas for helpful correspondence and conversations.  The author would also like to extend a big gratitude to Yves de Cornulier for the idea inspiring the examples in Section~\ref{sec.examples}.  Finally, the author would like to thank the anonymous referee for their extremely helpful comments which have greatly improved the exposition of this article.

\section{Definitions}\label{sec.HHG.defns}

In this section we will give the relevant background on HHSs and HHGs for our endeavours.  The definitions are rather technical so we will only focus on what we need, for a full account the reader should consult \cite[Definition~1.1, 1.21]{BehrstockHagenSisto2019}.  We will follow the treatment in \cite[Section~2]{PetytSpriano2020}.  To this end, a \emph{hierarchically hyperbolic space (HHS)} is pair $(X,\fraks)$ where $X$ is an $\epsilon$-quasigeodesic space and $\fraks$ is a set with some extra data which essentially functions as a coordinate system on $X$ where each coordinate entry is a hyperbolic space.  The relevant parts of the axiomatic formalisation are described as follows:
\begin{itemize}
    \item For each \emph{domain} $U\in\fraks$, there is a hyperbolic space $\calc U$ and \emph{projection} $\pi_U:X\to \calc U$ that is coarsely Lipschitz and coarsely onto \cite[Remark~1.3]{BehrstockHagenSisto2019}.
    \item $\fraks$ has a partial order $\nest$, called \emph{nesting}.  Nesting chains are uniformly finite, and the length of the longest such chain is called the \emph{complexity} of $(X,\fraks)$.
    \item $\fraks$ has a symmetric relation $\bot$, called \emph{orthogonality}.  The complexity bounds pairwise orthogonal sets of domains.
    \item The relations $\nest$ and $\bot$ are mutually exclusive.  The complement of $\nest$ and $\bot$ is called \emph{transversality} and denoted $\trans$.
    \item Whenever $U\trans V$ or $U\nest V$ there is a bounded set $\rho^U_V\subset \calc V$.  These sets, and projections of elements $x\in X$, are \emph{consistent} in the following sense:
    \begin{itemize}
        \item \emph{$\rho$-consistency:} Let $U,V,W\in\fraks$ such that $U\pnest V$ and $\rho_W^V$ is defined, then $\rho^U_W$ coarsely agrees with $\rho^V_W$;
        \item If $U\trans V$ then $\min\{d_{\calc U}(\pi_U(x),\rho^V_U),d_{\calc V}(\pi_V(x),\rho^U_V)\}$ is bounded.
    \end{itemize}
\end{itemize}
All coarseness may taken to be uniform so we can and will fix a uniform constant $\epsilon$ \cite[Remark~1.6]{BehrstockHagenSisto2019}. 

We remind the reader that these axioms for an HHS are not a complete set but only recall the structure we will need.  For the full definition the reader should consult \cite[Definition~1.1, 1.21]{BehrstockHagenSisto2019}.  The following definition of an HHG is however complete.

Let $X$ be the Cayley graph of a group $\Gamma$ and suppose $(X,\fraks)$ is an HHS, then $(\Gamma,\fraks)$ is a \emph{hierarchically hyperbolic group structure (HHG)} if it also satisfies the following:
\begin{enumerate}
    \item $\Gamma$ acts cofinitely on $\fraks$ and the action preserves the three relations.  For each $g\in G$ and each $U\in\fraks$, there is an isometry $g:\calc U\to \calc gU$ and these isometries satisfy $g\cdot h=gh$;
    \item for all $U,V\in\fraks$ with $U\trans V$ or $V\nest U$ and all $g,x\in\Gamma$ there is equivariance of the form $g\pi_U(x)=\pi_{gU}(gx)$ and $g\rho^V_U(x)=\rho^{gV}_{gU}(x)$.
\end{enumerate}
Note that this is not the original definition of an HHG as given in \cite{BehrstockHagenSisto2019}.  Instead, we have adopted the simpler axioms from \cite{PetytSpriano2020}.  Specifically, the axioms we have given require the equivariance to be exact rather than coarse and so imply the original axioms.  However, by \cite[Section~2.1]{DurhamHagenSisto2017Errata} if the axioms given in \cite{BehrstockHagenSisto2019} are satisfied then one can modify the HHG structure to satisfy the axioms given here.

\section{Hierarchical hyperbolicity and products} 
In this section we provide a proof of the folklore result that a group acting geometrically on a product of $\delta$-hyperbolic spaces with equivariant projections and without permuting isometric factors is an HHG. Let $X$ be a proper metric space and let $H=\Isom(X)$, then $H$ is a locally compact group with the topology given by uniform convergence on compacta.  Let $\Gamma$ be a discrete subgroup of $H$.  We say $\Gamma$ is a \emph{uniform lattice} if $X/\Gamma$ is compact.

\begin{prop}\label{prop.HHG}
Let $m>0$, $n\geq0$ and let $H\leq\Isom(\EE^n)\times\prod_{i=1}^m\Isom(X_i)$ be a closed subgroup, where each $X_i$ is a proper non-elementary $\delta$-hyperbolic space.  Let $\Gamma$ be a uniform $H$-lattice.  Suppose the projection $\pi_{\OO(n)}:\Gamma\to\OO(n)<\Isom(\EE^n)$ is trivial, then $\Gamma$ is a hierarchically hyperbolic group.
\end{prop}

\begin{proof}
Let $q$ be a $\Gamma$-equivariant quasi-isometry $\rm{Cay}(\Gamma,A)\to X$ given by the \v{S}varc-Milnor Lemma \cite[I.8.19]{BridsonHaefligerBook}.  If $n>0$, then for $j\in\{1-n,\dots,0\}$ let $X_j=\EE$ and $H_j=\Isom(\EE)$.  If $n>0$, then let $i\in\{1-n,\dots,m\}$, otherwise let $i\in\{1,\dots,m\}$.  Let $\fraks$ be the HHS structure for the product $X=\prod_{i=1-n}^mX_i$ given by \cite[Proposition~8.27]{BehrstockHagenSisto2019}.  As explained in the proof thereof every domain of $\fraks$ is either some $X_i$ for $i\in\{1-n,\dots,m\}$ or bounded (in fact a point) and labelled by $I\subseteq\{1-n,\dots,m\}$ corresponding to some non-trivial subproduct of $X$ with at least two factors.  The transversality relation is given by pairs $\{J,K\}$ of subsets of $I$ with $|J|,|K|\geq2$, and $J\cap K\neq\emptyset$. The nesting relation is given by inclusions of subproducts of $X$, and every distinctly labelled pair of domains which are not nested are orthogonal.

Note that $\fraks$ is finite and the action on $\fraks$ is trivial because $\Gamma$ does not permute isometric factors of $X$.  Indeed, $H$ which contains $\Gamma$ preserves the decomposition of $X$ and $\pi_{\OO(n)}(\Gamma)$ is trivial.  Every domain of the structure is a point or one of the $X_i$.  In the first case the $\Gamma$ action is trivial and in the second case $\Gamma$ acts via $\pi_{H_i}:\Gamma\to\Isom(X_i)$.  This immediately yields the first axiom because $\pi_{H_i}$ is a homomorphism.  The other $\rho$-consistency equivariance condition is established immediately since any two domains that are not points are orthogonal to each other.

For the second axiom consider the following diagram where the vertical arrows are given by applying the obvious group action:
\[ 
\begin{tikzcd}
{\Gamma\times\rm{Cay}(\Gamma,A)} \arrow[d] \arrow[rr, "{(\pi_{H_i},\pi_{X_i}\circ q)}"] &  & \pi_{H_i}(\Gamma)\times X_i \arrow[d] \\
{\rm{Cay}(\Gamma,A)} \arrow[rr, "\pi_{X_i}\circ q"']                                    &  & X_i .                                
\end{tikzcd}
\]
We will verify the diagram commutes.  Let $x\in\rm{Cay}(\Gamma,A)$ and $g\in \Gamma$.  First, we evaluate the composite map going down then across, we have 
\[(g,x)\mapsto gx\mapsto \pi_{X_i}(q(gx)).\]  
Going the other way we have 
\[(g,x)\mapsto (\pi_{H_i}(g),\pi_{X_i}(q(x)))\mapsto\pi_{X_i}(gq(x))=\pi_{X_i}(q(gx))\]
where the last equality is given by the $\Gamma$-equivariance of $q$. In particular, $g\pi_{X_i}(x)=\pi_{gX_i}(gx)=\pi_{X_i}(gx)$.
\end{proof}

\begin{lemma}\label{lem.HHG.FiniteExt}
If $\Gamma$ is a finite-by-(hierarchically hyperbolic group), then $\Gamma$ is a hierarchically hyperbolic group.
\end{lemma}
\begin{proof}
The group $\Gamma$ splits as a short exact sequence
\begin{equation}\label{eqn.ses}
\{1\}\rightarrowtail F\rightarrowtail\Gamma\twoheadrightarrow \Lambda\twoheadrightarrow\{1\}, 
\end{equation}
where $\Lambda$ is a hierarchically hyperbolic group and $F$ is the finite kernel of the action on the HHS $(\Lambda,\fraks)$.  Since $F$ acts trivially on $X$, it acts trivially on the HHG structure $\fraks$ for $\Lambda$.  The epimorphism $\varphi:\Gamma\twoheadrightarrow\Lambda$ induces an equivariant quasi-isometry $\psi$ on the associated Cayley graphs.  Thus, we may precompose every map in the HHG structure with $\varphi$ or $\psi$ to endow $\Gamma$ with the structure of an HHG.
\end{proof}

We restate Proposition~\ref{prop.HHG} in terms of groups acting geometrically on products of $\CAT(-1)$ spaces.  For an introduction to $\CAT(\kappa)$ groups and spaces see \cite{BridsonHaefligerBook}.  We will assume some non-degeneracy conditions on the $\CAT(0)$ spaces to avoid many technical difficulties associated with the $\CAT(0)$ condition (see \cite[Section~1.B]{CapraceMonod2009a} for a thorough explanation).  A group $H$ acting on a $\CAT(0)$ space $X$ is \emph{minimal} if there is no $H$-invariant closed convex subset $X'\subset X$.  If $\Isom(X)$ is minimal, then we say $X$ is minimal.

\begin{corollary}\label{cor.HHGCAT}
Let $\Gamma$ be a group acting properly cocompactly by isometries on a finite product of proper minimal $\CAT(-1)$-spaces.  If $\Gamma$ does not permute isometric factors, then $\Gamma$ is a hierarchically hyperbolic group.
\end{corollary}
\begin{proof}
The group $\Gamma$ splits as in \eqref{eqn.ses} where $\Lambda$ acts geometrically on a finite product of proper minimal $\CAT(-1)$-spaces and $F$ is the finite kernel of the action.  By Proposition~\ref{prop.HHG} we see that $\Lambda$ is an HHG and so by Lemma~\ref{lem.HHG.FiniteExt} $\Gamma$ is an HHG as well.
\end{proof}

The author suspects it is possible to strengthen the corollary to allow for permuting isometric factors provided the projection of $\Gamma$ to $\OO(n)<\Isom(\EE^n)$ is contained in $\OO_n(\ZZ)$.  To prove a converse to this corollary one may need to investigate the commensurators of maximal abelian subgroups of a hierarchically hyperbolic group $\Gamma$.  Indeed, the $\CAT(0)$ not biautomatic groups introduced by Leary--Minasyan \cite{LearyMinasyan2019} and the groups constructed by the author in \cite{Hughes2021a} (see also \cite{Hughes2021c} and \cite{HughesThesis}) have undistorted maximal abelian subgroups which have infinite index in their commensurator and are not virtually normal.  All of these groups have a non-discrete projection to $\OO(n)$.

\begin{question}
Is a maximal abelian subgroup $A$ of a hierarchically hyperbolic group $\Gamma$ either finite index in its commensurator $\Comm_\Gamma(A)$ or virtually normal?
\end{question}

\section{Non-virtually torsion-free lattices} \label{sec.examples}

In this section we will construct a cocompact lattice in a product of trees which is not virtually torsion-free.

Let $\Lambda$ be a Burger-Mozes simple group \cite{BurgerMozes1997,BurgerMozes2000struct,BurgerMozes2000lat} acting on $\calt_1\times\calt_2$ splitting as an amalgamated free product $F_n\ast_{F_m}F_n$ with embeddings $i,j:F_m\to F_n$.  This defines a group $\Lambda$ which embeds discretely into the product of $T_1=\Aut(\calt_1)$ and $T_2=\Aut(\calt_2)$ with compact quotient.  For instance one may take Rattaggi's example of a lattice in the product of an $8$-regular and $12$-regular tree which splits as $F_7\ast_{F_{73}}F_7$  \cite{Rattaggi2007a} (see also \cite{Rattaggi2007b}) or one of Radu's examples \cite{Radu2020}.

Define $A=\ZZ_p\rtimes F_n$ for $p$ prime such that the $F_n$-action is non-trivial\footnote{The key point here is that $F_n$ will normally generate $A$, in particular, other finite groups with non-trivial $F_n$-action could be used here.}.  Consider the embeddings $\widetilde{i},\widetilde{j}:F_m\rightarrowtail F_n\rightarrowtail A$ given by the composition of $i$ or $j$ with the obvious inclusion.  Now, we build a group $\Gamma$ as an amalgamated free product $A\ast_{F_m}A$, note that $\Gamma$ surjects onto the original Burger-Mozes group $\Lambda$ with kernel the normal closure of the torsion elements. Let $\calt_3$ denote the Bass-Serre tree of $\Gamma$ and let $T_3$ denote the corresponding automorphism group.

\begin{prop}\label{prop.ex.lat}
$\Gamma$ is a uniform $(T_1\times T_3)$-lattice which does not permute the factors.
\end{prop}
This can be easily deduced by endowing $\Gamma$ with a graph of lattices structure in the sense of \cite[Definition~3.2]{Hughes2021a} and then applying \cite[Theorem~A]{Hughes2021a}.  Instead we will provide a direct proof.
\begin{proof}
The group $\Gamma$ acts on its Bass-Serre tree $\calt_3$ and also on $\calt_1$ via the homomorphism $\psi:\Gamma\to T_1$ defined by taking the composition of the surjection $\Gamma\twoheadrightarrow\Lambda$ and the projection $T_1\times T_2\to T_2$.  The diagonal action on the product space $\calt_1\times \calt_3$ is properly discontinuous cocompact and by isometries.  Indeed, the action is clearly cocompact since $\calt_1/\Gamma=\calt_1/\Lambda$ is a finite graph and $\calt_3/\Gamma$ is a finite graph by construction.  The action is properly discontinuous since by construction the only elements which fix a point in $\calt_1\times\calt_3$ are finite order and every torsion subgroup is finite. The kernel of the action is trivial, since the only elements which could act trivially are the torsion elements.  However, these all clearly act non-trivially on $\calt_3$ by elementary Bass-Serre theory.  Thus, the action is faithful.  We conclude that $\Gamma$ is a uniform $(T_1\times T_3)$-lattice.
\end{proof}

It remains to show $\Gamma$ is not virtually torsion-free.  

\begin{prop}\label{prop.ex.tf}
The group $\Gamma$ has no proper finite-index subgroups and contains torsion.  In particular, $\Gamma$ is not virtually torsion-free.
\end{prop}

The author thanks Yves de Cornulier for the following argument.

\begin{proof}
Note that since $F_n$ acts non-trivially on $\ZZ_p$ it follows that $F_n$ normally generates $A$.  Because the Burger--Mozes subgroup $\Lambda < \Gamma$ is a simple group, every finite index normal subgroup of $\Gamma$ contains it. Thus, their intersection $\Gamma^{(\infty)}=\bigcap_{[\Gamma:\Gamma']<\infty}\Gamma'$ contains $\Lambda$.  It follows, both copies of $F_n$ are contained in $\Gamma^{(\infty)}$.  Now, $F_n$ normally generates $A$, so $\Gamma^{(\infty)}=\Gamma$.  In particular, $\Gamma$ has no proper finite-index subgroups.  Since, $A$ is not torsion-free, we conclude that $\Gamma$ is not virtually torsion-free.
\end{proof}

To summarise we have the following theorem.

\begin{thm}[Theorem~\ref{corx.HHGsNotVirTorFree}]\label{cor.HHGsNotVirTorFree}
The group $\Gamma$ is a cocompact lattice in a product of trees, is a hierarchically hyperbolic group, and is not virtually torsion-free.
\end{thm}
\begin{proof}
By Proposition~\ref{prop.ex.lat} and Corollary~\ref{cor.HHGCAT} we see that $\Gamma$ is a hierarchically hyperbolic group.  By Proposition~\ref{prop.ex.tf} we see $\Gamma$ is not virtually torsion-free.
\end{proof}

\begin{remark}
In \cite[Corollary~8.7]{Hughes2021a} the author gave a way to use A.~Thomas's construction in \cite{Thomas2006} to promote lattices in products of trees to lattices in products of ``sufficiently symmetric'' right-angled buildings.  Applying \cite[Corollary~8.7]{Hughes2021a} to one of the non-virtually torsion-free lattices $\Gamma$ we obtain a non-virtually torsion-free lattice $\Lambda$ acting on a product of ``sufficiently symmetric'' right-angled hyperbolic buildings each not quasi-isometric to a tree.  Moreover, by Corollary~\ref{cor.HHGCAT} $\Lambda$ is hierarchically hyperbolic.
\end{remark}

\AtNextBibliography{\small}
\printbibliography

\end{document}